\documentclass[11pt]{amsart}
\usepackage{amssymb}
\usepackage[T1]{fontenc}
\usepackage{hyperref}

\usepackage{anysize}
\marginsize{2.5cm}{2.5cm}{2.5cm}{2.5cm}
\theoremstyle{plain}
\newtheorem{theorem}{Theorem}
\newtheorem{proposition}{Proposition}
\newtheorem{lemma}{Lemma}

\theoremstyle{remark}
\newtheorem{remark}{Remark}

\newcommand{\Q}{\mathbb{Q}}

\title{There are infinitely many rational Diophantine sextuples with square denominators}

\author{Andrej Dujella}
\address{Department of Mathematics, University of Zagreb, Bijeni\v{c}ka cesta 30, 10000 Zagreb, Croatia}
\email{duje@math.hr}
\author{Matija Kazalicki}
\address{Department of Mathematics, University of Zagreb, Bijeni\v{c}ka cesta 30, 10000 Zagreb, Croatia}
\email{matija.kazalicki@math.hr}
\author{Vinko Petri\v cevi\'c}
\address{Department of Mathematics, University of Zagreb, Bijeni\v{c}ka cesta 30, 10000 Zagreb, Croatia}
\email{vpetrice@math.hr}



\thanks{The authors were supported by the QuantiXLie Centre of Excellence, a project co-financed by the Croatian Government and European Union through the European Regional Development Fund - the Competitiveness and Cohesion Operational Programme (Grant KK.01.1.1.01.0004), and by the Croatian Science Foundation under the project no. IP-2018-01-1313.}

\begin{document}

\begin{abstract}
A rational Diophantine $m$-tuple is a set of $m$ nonzero rationals such that the product of any two of them increased by $1$ is a perfect square. The first rational Diophantine quadruple was found by Diophantus, while Euler proved that there are infinitely many rational Diophantine quintuples. In 1999, Gibbs found the first example of a rational Diophantine sextuple, and in 2016 Dujella, Kazalicki, Miki\'c and Szikszai proved that there are infinitely many of them. In this paper, we prove that there exist infinitely many rational Diophantine sextuples such that the denominators of all the elements in the sextuples are perfect squares.
\end{abstract}

\maketitle

\section{Introduction}

\noindent A set of $m$ nonzero rationals $\{a_1,a_2,\ldots,a_m\}$ is called {\em a rational Diophantine $m$-tuple} if $a_ia_j+1$ is a perfect square for all $1\leq i<j\leq m$.  The first example of a rational Diophantine quadruple was the set
$$
\left\{\frac{1}{16},\, \frac{33}{16},\, \frac{17}{4},\, \frac{105}{16}\right\}
$$
found by Diophantus (see \cite{Dio}). Euler found infinitely many rational Diophantine quintuples (see \cite{Hea}),  e.g. he was able to extend the integer Diophantine quadruple
$$
\{1,3,8,120\}
$$
found by Fermat, to the rational quintuple
$$
\left\{ 1, 3, 8, 120, \frac{777480}{8288641} \right\}.
$$
Stoll \cite{S} recently showed that this extension is unique.

\noindent In 1999, Gibbs found the first example of rational Diophantine sextuple \cite{Gibbs1}
$$
\left\{ \frac{11}{192}, \frac{35}{192}, \frac{155}{27}, \frac{512}{27}, \frac{1235}{48}, \frac{180873}{16} \right\},
$$
and in 2016 Dujella, Kazalicki, Miki\'c and Szikszai \cite{D-K-M-S} showed that there are infinitely many rational Diophantine triples that can be extended to the Diophantine sextuple in infinitely many ways.
For example, there are infinitely many rational Diophantine sextuples
containing the triples
$\{15/14, -16/21, 7/6\}$ and $\{ 3780/73, 26645/252, 7/13140\}$.
Soon after that, Dujella and Kazalicki \cite{D-K} (inspired by the work of Piezas \cite{P}) described another construction of rational Diophantine sextuples extending an infinite class of Diophantine quadruples to Diophantine sextuples in one way. For the description of this family see Section \ref{sec:1}.

No example of a rational Diophantine septuple is known. On the other hand, as a consequence of the Lang conjecture on varieties of general type we expect the number of elements of a rational Diophantine tuple to be bounded (see Introduction of \cite{D-K-M-S}). For more information on Diophantine $m$-tuples see the survey article \cite{D}.

In this paper we study arithmetic properties of rational Diophantine sextuples, in particular we prove the following theorem.

\begin{theorem}
There are infinitely many rational Diophantine sextuples such that denominators of all the elements (in the lowest terms) in the sextuples are perfect squares.
\end{theorem}

We can describe one such family in the following way. Let $C:t^4-1=-6s^2$ be a genus one curve defined over $\Q$. It is birationally equivalent to the elliptic curve $E':y^2=x^3+9x$. Denote by $T=[4,10]$ the point of infinite order in Mordell-Weil group $E'(\Q)$. For a positive integer $k$ denote by $t_k$ the $t$-coordinate of the point on $C$ that corresponds to the point $[k]T$ on $E'(\Q)$ under birational equivalence.
Let $\mathcal{F}(t)$ be a family of Diophantine sextuples (for the construction of this family see Section \ref{sec:1})

\begin{eqnarray*}
&\{-\frac{9 \left(t^2+1\right)}{8 (t-1) (t+1)},\frac{\left(t^2-7\right) \left(t^2+1\right) \left(7
   t^2-1\right)}{8 (t-1)^3 (t+1)^3},\frac{8 (t-1)^3 (t+1)^3}{9 \left(t^2+1\right)^3},-\frac{2
   \left(t^2+5\right) \left(5 t^2+1\right)}{9 (t-1) (t+1) \left(t^2+1\right)},\\ &-\frac{2 t \left(t^2-4
   t-3\right) \left(3 t^2-4 t-1\right) \left(t^3+8 t^2+5 t+4\right) \left(4 t^3-5 t^2+8 t-1\right)}{(t-1)
   (t+1) \left(t^2+1\right) \left(t^4+34 t^2+1\right)^2},\\ &\frac{2 t \left(t^2+4 t-3\right) \left(3 t^2+4
   t-1\right) \left(t^3-8 t^2+5 t-4\right) \left(4 t^3+5 t^2+8 t+1\right)}{(t-1) (t+1) \left(t^2+1\right)
   \left(t^4+34 t^2+1\right)^2} \}\enspace.
\end{eqnarray*}

\begin{proposition}\label{prop:1}
	If $k$ is a positive integer such that $k\equiv 1,2 \pmod{3}$, then $\mathcal{F}(t_k)$ is a rational Diophantine sextuple such that denominators of all the elements in the sextuple are perfect squares.
\end{proposition}
\begin{remark}
Our construction always produced the sextuples with the mixed signs since the product of the first and the third element is a negative number.
\end{remark}

\section{Interpolating numerical data}\label{sec:1}

Our starting point is an example of Diophantine sextuple with square denominators
$$l=\left\{\frac{75}{8^2},-\frac{3325}{64^2},-\frac{12288}{125^2},\frac{123}{10^2},
\frac{3498523}{2260^2},\frac{698523}{2260^2}\right\},$$
which we have discovered  (together with seventeen other examples) by a numerical search which we now briefly describe.

In the first step of this experiment, we generated all Diophantine quintuples with numerators and denominators in the range between $-2^{15}$ and $2^{15}$.
Next, we tried to extend each quintuple to a Diophantine sextuple by requiring that the sixth element forms a regular $m$-tuple (where $m=3,4,5$ or $6$) with some elements from that quintuple. E.g. if $\{a,b,c,d,e\}$ is one such Diophantine quintuple, we can define $f$ such that $(a+b-f-c)^2 = 4(ab + 1)(fc + 1)$ holds (i.e. such that $\{a,b,f,c\}$ is a regular quadruple), and then check if $\{a,b,c,d,e,f\}$ is a Diophantine sextuple. Justification for this heuristics comes from the observation \cite{G3} that a ``random'' Diophantine sextuples often contains a regular Diohantine $m$-tuple. For more information about regular $m$-tuples see \cite{G2}.

In order to describe a family of Diophantine sextuples containing $l$, we recall the construction from \cite{D-K}.

Let $\{a, b, c, d\}$ be a rational Diophantine quadruple such that 
\begin{equation}\label{eq:1}
(abcd-3)^2=4(ab+cd+3),
\end{equation}
and let $x_1$ and $x_2$ be the roots of $$(abcdx+2abc+a+b+c-d-x)^2=4(ab+1)(ac+1)(bc+1)(dx+1).$$
If $x_1x_2 \ne 0$ then Proposition 1 in \cite{D-K} (see also \cite{P}) states that $\{a,b,c,d,x_1,x_2\}$ is a Diophantine sextuple.

Since the quadruple $(a,b,c,d)$ satisfies
\begin{align*}
ab+1=t_{12}^2 \quad ac+1&=t_{13}^2 \quad ad+1=t_{14}^2\\
bc+1=t_{23}^2 \quad bd+1&=t_{24}^2 \quad cd+1=t_{34}^2,
\end{align*}
where $t_{ij}$ are in $\Q$ or $\Q(t)$, it follows that $(t_{12},t_{34},t_{13},t_{24},t_{14},t_{23}, m'=abcd)$ defines a rational point on an algebraic variety $\mathcal{C}$ defined by the following equations:
\begin{align*}
(t_{12}^2-1)(t_{34}^2-1)&=m'\\
(t_{13}^2-1)(t_{24}^2-1)&=m'\\
(t_{14}^2-1)(t_{23}^2-1)&=m'.
\end{align*}

Conversely, the points $(\pm t_{12},\pm t_{34},\pm t_{13},\pm t_{24},\pm t_{14},\pm t_{23}, m')$ on $\mathcal{C}$ determine two rational Diophantine quadruples $\pm(a,b,c,d)$ (for example $a^2=(t_{12}^2-1)(t_{13}^2-1)/(t_{23}^2-1)$) provided that the elements $a,b,c$ and $d$ are rational, distinct and non-zero.
(Note that if one element is rational, then all the elements are rational.)

The projection $(t_{12},t_{34},t_{13},t_{24},t_{14},t_{23}, m')\mapsto m'$ defines a fibration of $\mathcal{C}$ over the projective line, and a generic fiber is the product of three genus one curves $\mathcal{D}:(x^2-1)(y^2-1)=m'$, hence any point on $\mathcal{C}$ corresponds to the three points $Q_1=(t_{12},t_{34})$, $Q_2=(t_{13},t_{24})$ and $Q_3=(t_{14}, t_{23})$ on $\mathcal{D}$. The condition \eqref{eq:1} is equivalent to $t_{12} t_{34}= \pm t_{12} \pm t_{34}$, or $t_{34}=\pm t_{12}/(t_{12}\pm 1)$. Hence, if we set $t_{12} =t$, $t_{34}=t/(t-1)$ and $m'=(t^2-1)(\frac{t^2}{(t-1)^2}-1)=\frac{2t^2 + t - 1}{t - 1}$, the  condition \eqref{eq:1} is automatically satisfied.

The curve $\mathcal{D}$ over $\Q(t)$
$$\mathcal{D}: (x^2-1)(y^2-1)=\frac{2t^2 + t - 1}{t - 1}$$ is birationally equivalent to the elliptic curve
$$E: S^2 = T^3 -2\cdot \frac{2t^2 - t + 1}{t-1}T^2 + \frac{(2t - 1)^2 (t + 1)^2}{(t-1)^2} T.$$
The map is given by $T = 2(x^2-1)y+2x^2-(2-m')$,
and $S = 2Tx$, where $m'=\frac{2t^2 + t - 1}{t - 1}$.

Denote by $P=\displaystyle\left[\frac{(2t-1)^2(t+1)}{t-1}, \frac{2t(2t-1)^2(t+1)}{t-1}\right]\in E(\Q(t))$ a point of infinite order on $E$, and by $R
 = \displaystyle\left[\frac{(t+1)(2t-1)}{t-1}, \frac{2(t+1)(2t-1)}{t-1}\right]$ a point of order $4$. The point $(t_{12},t_{34}) \in \mathcal{D}(\Q(t))$ corresponds to the point $P\in E(\Q(t))$. The points $P$ and $R$ generate Mordell-Weil group $E(\Q(t))$ (see \cite{D-K}).

Therefore, to specify the family of Diophantine quadruples which satisfies \eqref{eq:1} and whose product is $\frac{2t^2 + t - 1}{t - 1}$, we need to specify two points $Q_2$ and $Q_3$ in $E(\Q(t))$ (\cite{D-K} gives sufficient conditions for three points $P,Q_2$ and $Q_3$ in $E(\Q(t))$ to define the quadruple whose elements are rational, distinct and non-zero).   

If we go back to our example, we can observe that the first four elements of $l$ satisfy condition $\eqref{eq:1}$, and that their product is equal to $\frac{2t_0^2 + t_0 - 1}{t_0 - 1}$ where $t_0=113/625$. Inspired by the fact that $\frac{1-t_0}{2}=16/25$ is a square, we restrict to the subfamily $t:=1-2t^2$, i.e. we consider the base change of $E/\Q(t)$ to $\Q\left(\sqrt{\frac{1-t}{2}}\,\right)$. If we further specialize $t:=\frac{3t^2+3}{4t^2-4}$ and denote the resulting elliptic curve also by $E$, then we obtain another point of infinite order in $E(\Q(t))$ $$Q=\left[-\frac{\left(t^2+5\right)^2 \left(5 t^2+1\right)^2}{36 (t-1)^2 (t+1)^2
   \left(t^2+1\right)^2},\frac{t \left(t^2+5\right)^2 \left(5 t^2+1\right)^2}{9 (t-1)^2 (t+1)^2
   \left(t^2+1\right)^3}\right].$$
	
It is easy to check that the family $\mathcal{F}$ corresponds to the triple $(P,Q,P+R)\in E(\Q(t))^3$, and we obtain $l$ if we specialize to $t=-1/7$, i.e. $l=\mathcal{F}(-1/7)$.

\section{Analysis of the family $\mathcal{F}$}
For a prime $p$ and $q\in \Q$, we denote by $v_p(q)$ a $p$-adic valuation of $q$ normalized such that $v_p(p)=1$. Let $E':y^2=x^3+9x$ be an elliptic curve birationally equivalent to $C$ via the map $h:E'\rightarrow C$, $h(x,y)=\left(\frac{3-x}{3+x},\frac{-2y}{(x+3)^2}\right)$. Let $T=[4,10] \in E'(\Q)$ be a generator of the free part of the Mordell-Weil group $E'(\Q)$. For an integer $k$, denote by $(t_k,s_k):=h([k]T)\in C$. 
\begin{lemma}\label{lemma:1} Let $k$ be a positive integer.
If $k\equiv 0 \pmod{3}$ then $v_3(s_k)>0$ and $v_3(t_k - 2)\ge 1$. If $k\equiv 1,2 \pmod{3}$  then $v_3(t_k-5)\ge 2$ and $v_3(s_k)=0$.
\end{lemma}
\begin{proof}
Define $(x_k,y_k):=[k]T$. Since $y_k^2=x_k^3+9x_k$, it follows that either $v_3(x_k)<0$ or $x_k$ is a square mod $9$.
The first case occurs when $k\equiv 0 \pmod{3}$ (since $[3]T$ is in the kernel of mod $3$ reduction map on $E'$), and then $t_k=\frac{3-x_k}{3+x_k}$ implies that $v_3(t_k+1)\ge 2$.
In the second case one finds that $v_3(t_k-5)\ge 2$ (since the squares mod $9$ are $1,4$ and $7$).
In the first case, since $t_k^4-1=-6 s_k^2$ it follows $v_3(s_k)>0$, while in the second case $v_3(s_k)=0$.
\end{proof}

\begin{proof}[Proof of Proposition \ref{prop:1}]

	We analyze denominators of all the elements in $\mathcal{F}=\{a_1(t),a_2(t),\ldots,a_6(t)\}$ separately.
	\begin{enumerate}
		\item [i)] Since $a_1(t)=-\frac{9 \left(t^2+1\right)}{8 (t-1) (t+1)}=-\frac{9 \left(t^2+1\right)^2}{8 \left(t^4-1\right)}$, we have that $a_1(t_k)=\frac{3 \left(t_k^2+1\right)^2}{16 s_k^2}$.  Similarly, $a_3(t)=\frac{8 (t-1)^3 (t+1)^3}{9 \left(t^2+1\right)^3}=\frac{8 \left(t^4-1\right)^3}{9 \left(t^2+1\right)^6}$, and $a_3(t_k)=\frac{-3\cdot 2^6 s_k^6}{(t_k^2+1)^6}$.
		The claim follows from Lemma \ref{lemma:1} since it implies that $v_3(s_k)=0$ if $k\equiv 1,2 \pmod{3}$.
		\item[ii)] We have that $a_2(t)=\frac{\left(t^2-7\right) \left(t^2+1\right) \left(7 t^2-1\right)}{8 (t-1)^3 (t+1)^3}=\frac{\left(t^2-7\right) \left(t^2+1\right)^4 \left(7 t^2-1\right)}{8 \left(t^4-1\right)^3}$.  The only primes that can divide denominator of $a_2(t_k)=\frac{-\left(t_k^2-7\right) \left(t_k^2+1\right)^4 \left(7 t_k^2-1\right)}{3\cdot 24^2 s_k^6}$ are the primes that divide $s_k$ or $1/t_k$. Assume that $p$ is a prime different than $2$ and $3$. Since $t_k^4-1=-6 s_k^2$, if $v_p(t_k)<0$, then $v_p(s_k)=2v_p(t_k)$. Hence for such $p$, $v_p(a_2(t_k))= 1$ if $p=7$ and $0$ otherwise. If $v_p(s_k)>0$, then $v_p(a_2(t_k))$ is even unless $v_p((t_k^2-7)(7 t_k^2-1))>0$. Since the resultant of polynomials $t^4-1$ and $(t^2-7)(7t^2-1)$ is equal to $2^{16} 3^4$ (and is divisible only by $2$ and $3$) this can not happen.
		To rule out $p=3$ case, we note that $v_3\left((t_k^2-7)(7 t_k^2-1)\right)\ge 3$ since Lemma \ref{lemma:1} implies that $v_3(t_k - 5)\ge 2$ if $k \equiv 1,2 \pmod{3}$. Hence $v_3(a_2(t_k)) \ge 0$. If $v_2(s_k)>0$, then $v_2(t_k)=0$ and $v_2((t_k^2-7)(7 t_k^2-1))=2$ is even number. Finally, if $v_2(t_k)<0$, then $4v_2(t_k)=2v_2(s_k)+1$ which is not possible.
\item[iii)]
Since $a_4(t)=-\frac{2 \left(t^2+5\right) \left(5 t^2+1\right)}{9 \left(t^4-1\right)}$ then $a_4(t_k)=\frac{ \left(t_k^2+5\right) \left(5 t_k^2+1\right)}{3\cdot 3^2 s_k^2}$. Let $p$ be a prime different than $2$ and $3$. Same as in the part ii), if $v_p(t_k)<0$ then $v_p(a_4(t_k))=1$ if $p=5$ and $0$ otherwise.
If $v_p(s_k)>0$, then $v_p(a_4(t_k))$ is even unless $v_p\left(\left(t_k^2+5\right) \left(5 t_k^2+1\right)\right)>0$. Since the resultant of polynomials $t^4-1$ and $\left(t^2+5\right) \left(5 t^2+1\right)$ is $2^{12}3^4$ this can not happen. Note that $v_3\left(\left(t_k^2+5\right) \left(5 t_k^2+1\right)\right)\ge 3$ (Lemma \ref{lemma:1}), hence if $k\equiv 1,2 \pmod{3}$ then $v_3(a_4(t_k) \ge 0$. As earlier, $v_2(t_k)<0$ is not possible, and if $v_2(s_k)>0$ then $v_2(t_k)=0$ and $v_2\left(\left(t_k^2+5\right) \left(5 t_k^2+1\right)\right)=2$.

\item[iv)] We have that $$a_5(t_k)=\frac{ t_k \left(t_k^2-4 t_k-3\right) \left(3 t_k^2-4 t_k-1\right) \left(t_k^3+8 t_k^2+5 t_k+4\right) \left(4 t_k^3-5t_k^2+8 t_k-1\right)}{3s_k^2 \left(t_k^4+34 t_k^2+1\right)^2}.$$ 
Let $p$ be a prime different than $2$ and $3$. If $v_p(t_k)<0$ then $v_p(a_5(t_k))=11 v_p(t_k)-8 v_p(t_k)-2 v_p(s_k)=-v_p(t_k)>0$. If $v_p(s_k)>0$, then $v_p(a_5(t_k))$ is even unless $$v_p\left( t_k \left(t_k^2-4 t_k-3\right) \left(3 t_k^2-4 t_k-1\right) \left(t_k^3+8 t_k^2+5 t_k+4\right) \left(4 t_k^3-5t_k^2+8 t_k-1\right)\right)>0.$$ As before, since the resultant is divisible only by $2$ and $3$ this can not happen. The resultant of polynomials $t \left(t^2-4 t-3\right) \left(3 t^2-4 t-1\right)\left(t^3+8 t^2+5 t+4\right) \left(4 t^3-5t^2+8 t-1\right)$ and $t^4+34t^2+1$ is $2^{44}3^{16}$, hence we have the same conclusion if $v_p(t_k^4+34t_k^2+1)>0$. As earlier, $v_2(t_k)<0$ is not possible. If $v_2(s_k)>0$, then $v_2(t_k)=0$, and we find that $$v_2\left( t_k \left(t_k^2-4 t_k-3\right) \left(3 t_k^2-4 t_k-1\right) \left(t_k^3+8 t_k^2+5 t_k+4\right) \left(4 t_k^3-5t_k^2+8 t_k-1\right)\right)=4.$$
Similarly, if $v_2(t_k^4+34t_k^2+1)>0$, then $v_2(t_k)=0$, and we have the same conclusion as above. 

If $k\equiv 1,2 \pmod{3}$, then since $v_3(t_k-5)>1$ by the direct calculation we can show that $v_3\left(3s_k^2 \left(t_k^4+34 t_k^2+1\right)^2\right)=5$. Similarly we can check that $v_2(t_k) = 0$, $v_3(t_k^2-4 t_k-3)=0$, $v_3(3 t_k^2-4 t_k-1)\ge 2$, $v_3(t_k^3+8 t_k^2+5 t_k+4)=1$ and $v_3(4 t_k^3-5t_k^2+8 t_k-1)=2$ which implies that $v_3(a_5(t_k))\ge 0$.

The claim for $a_6(t_k)$ is proved in a similar way.
 
\end{enumerate}
\end{proof}

\end{document}